\documentclass[10pt]{article}
\usepackage{amsmath,amssymb,amsthm,amsfonts,amscd,euscript,verbatim, t1enc, newlfont, xypic}
\usepackage{hyperref}
\hypersetup{breaklinks}

\theoremstyle{definition}
\newtheorem{theo}{Theorem}[subsection]

\newtheorem{pr}[theo]{Proposition}

 \newtheorem{coro}[theo]{Corollary}
  
\theoremstyle{remark}
\newtheorem{rema}[theo]{Remark}

\theoremstyle{definition}
\newtheorem{defi}[theo]{Definition}

\numberwithin{equation}{subsection}

\newcommand\ob{^{-1}}

\newcommand\dmge{DM^{eff}_{gm}{}}

\newcommand\cha{\operatorname{char}}

\newcommand\mg{M_{gm}}

\newcommand\obj{\operatorname{Obj}}

\newcommand\id{\operatorname{id}}
\newcommand\cu{\underline{C}}
\newcommand\du{\underline{D}}

\newcommand\com{\mathbb{C}}

\newcommand\z{{\mathbb{Z}}}
\newcommand\q{{\mathbb{Q}}}
\newcommand\af{\mathbb{A}}
 
\newcommand\p{\mathbb{P}}

\newcommand\ff{\mathbb{F}}

\newcommand\pt{\operatorname{pt}}
\newcommand\eps{\varepsilon}

\newcommand\lam{\Lambda}

\newcommand\ns{\{0\}}

\DeclareMathOperator\inli{\varinjlim}

\newcommand\chow{\operatorname{Chow}}

\newcommand\chowlam{\operatorname{Chow}_{\lam}}

\newcommand\sv{\operatorname{SmVar}}

\newcommand\spe{\operatorname{Spec}}

\newcommand\modd{\operatorname{Mod}}

\newcommand\zop{{\mathbb{Z}[\frac{1}{p}]}}

\newcommand\zoh{\z[\frac{1}{2}]}

\newcommand\pfi{\mathcal{PF}i}
\newcommand\sht{SH}
\newcommand\shtc{SH^c}

\newcommand\shtck{SH^c(k)}
\newcommand\shtk{SH(k)}

\newcommand\she{SH^{eff}}
\newcommand\dmeb{DM^{eff}}

\newcommand\dm{DM}
\newcommand\dmk{DM(k)}
\newcommand\dmkp{DM(k')}
\newcommand\dmc{DM^c}
\newcommand\dmck{DM^c(k)}

\newcommand\tsh{t_{hom}^{SH}}
\newcommand\tdm{t_{hom}^{DM}}
\newcommand\tshlam{t_{\lam}^{SH}}
\newcommand\tdmlam{t_{\lam}^{DM}}
\newcommand\tshelam{t_{\lam}^{SH^{eff}}}
\newcommand\tdmelam{t_{\lam}^{DM^{eff}}}

\newcommand\thom{t_{hom}}

\newcommand\msh{\sht(m)} 
\newcommand\mdm{\dm(m)} 
\newcommand\mshlam{\shlam(m)} 
\newcommand\mdmlam{\dmlam(m)} 

\newcommand\shtlamk{SH_{\lam}(k)}
\newcommand\dmlamk{DM_{\lam}(k)}
\newcommand\shtlam{SH_{\lam}}
\newcommand\dmlam{DM_{\lam}}

\newcommand\mgl{\operatorname{MGl}}
\newcommand\mglmod{\operatorname{MGl}-\modd}
\newcommand\shmgl{D^{\operatorname{MGl}}}

\newcommand\shmglam{D^{\mgl}_{\lam}}

\newcommand\emz{H_{\z}}
\newcommand\emlam{H_{\lam}}
\newcommand\emlamo{{H}'_{\lam}} 
\newcommand\filt{\operatorname{Fil}_{\operatorname{Tate}}}

\newcommand\shtohk{SH(k)[\frac{1}{2}]}
\newcommand\shtpl{SH^+}

\newcommand\shtplk{SH^+(k)}
\newcommand\shtmik{SH^-(k)}
\newcommand\mk{M_k}
\newcommand\mklam{M_{k,\lam}}
\newcommand\mklamc{M^c_{k,\lam}}

\newcommand\sinft{\Sigma_T^{\infty}}
\newcommand\sinftlam{\Sigma_{T,\lam}^{\infty}}
\newcommand\mglam{M_{gm,\lam}}
\newcommand \shelam{SH^{eff}_{\lam}}
\newcommand \shelamc{SH^{eff,c}_{\lam}}

\newcommand \dmelam{DM^{eff}_{\lam}}
\newcommand \dmelamc{DM^{eff,c}_{\lam}}
\newcommand \shlam{SH_{\lam}}
\newcommand \shlamk{SH_{\lam}(k)}
\newcommand \shlamc{SH^c_{\lam}}
\newcommand\dmlamc{\dm_{\lam}^c}

\newcommand\afo{\mathbb{A}^1}

\DeclareMathOperator\imm{\operatorname{Im}}
\DeclareMathOperator\co{\operatorname{Cone}}

\newcommand\hrt{{\underline{Ht}}}

\begin{document}

 \title{On  infinite effectivity of motivic spectra and the vanishing of their motives} 
  \author{M.V. Bondarko
   \thanks{Research is supported by the Russian Science Foundation grant no. 16-11-10200.}}\maketitle

\begin{abstract}
We study  the kernel of the "compact motivization" functor $M_{k,\lam}^c:\shlamc(k)\to \dmlamc(k)$ (i.e., we try to  describe those compact objects of the $\lam$-linear version of $SH(k)$ whose associated motives vanish; here $\z\subset \lam \subset \q$). We also investigate the question when the $0$-homotopy connectivity of $M^c_{k,\lam}(E)$ 
ensures the $0$-homotopy connectivity of $E$ itself 
(with respect to the homotopy $t$-structure $\tshlam$ for $\shlam(k)$). We prove that 
 the kernel of $\mklamc$  vanishes  and the corresponding "homotopy connectivity detection" statement is also valid if and only if $k$ is a non-orderable field; this is an easy consequence of 
 similar results of T. Bachmann (who considered the case where the cohomological $2$-dimension of $k$ is finite). 
Moreover, for an arbitrary $k$ 
 the kernel  in question does not contain any $2$-torsion 
(and the author also suspects that all its elements are odd torsion  unless $\frac{1}{2}\in \lam$).
Furthermore, if  the exponential characteristic of $k$ is invertible in $\lam$ then this kernel consists exactly of "infinitely effective" (in the sense of Voevodsky's slice filtration) objects of  $\shlamc(k)$. 

 The results and methods of this paper may 
 be useful for the study  of motivic spectra. In particular, they (combined with other ideas of Bachmann) imply  the tensor invertibility of  motivic spectra of affine quadrics over non-orderable fields.
 We also generalize a theorem of A. Asok.
\end{abstract}

\tableofcontents

 \section*{Introduction}

It is well known that for a perfect field $k$ both the (Morel-Voevodsky's) motivic stable homotopy category $\shtk$ and  Voevodsky's motivic category
$\dmk$ are important for the study of cohomology of $k$-varieties. The roles of these categories are somewhat distinct: whereas $\sht(k)$ is "closer to the geometry" of varieties, $\dmk$ is somewhat easier to deal with. For instance, we know much more on morphisms in  $\dmk$ than in $\shtk$; this information yields the existence of  so-called Chow weight structures on $\dmc(k)\subset \dmk$ (as shown in \cite{bws} and \cite{bzp}; below we will mention an interesting application of this result described in \cite{bachinv}). 

Now, there is a connecting functor $M_k:\sht(k)\to \dmk$ (that sends the motivic spectra of smooth varieties into their motives); so it is rather important to describe the extent to which $M_k$ is conservative. Whereas the "whole" $M_k$ is never conservative (as demonstrated in Remark \ref{rshmi}(1) below), 
 Theorem 16 of  \cite{bachcons} states that the restriction $M_k^c$ of $M_k$ to compact objects is conservative whenever $k$ is of finite cohomological $2$-dimension. 

The current paper grew out of the following  observation:  this theorem can be generalized to the case of  an
arbitrary non-orderable (perfect) 
$k$ via a simple "continuity" argument (i.e., if $k=\inli k_i$ then the conservativity of all $M_{k_i}^c$ implies that for $M_k^c$). We also demonstrate (in Remark \ref{rshmi}(3)) that this conservativity statement 
fails whenever $k$ is a formally real field (though we conjecture that the  
 kernel of $M_k^c$ consists of torsion elements only; see Remark \ref{rshmi}(4)); thus we answer the question when $M_k^c$ is conservative completely. Moreover, we  extend to arbitrary non-orderable fields the stronger part of  Bachmann's Theorem 16(b); 
 so we prove that the $r$-homotopy connectivity of $M_{k}^c(E)$ (for $E\in \obj \shtc(k)$) ensures the $r$-homotopy connectivity of $E$ itself (here $r$-homotopy connectivity means belonging to the $\thom\ge r+1$-part for the corresponding homotopy $t$-structure). 
Lemma 19 of ibid. also 
 gives a similar result for $E$ being a $2$-torsion compact motivic spectrum over an arbitrary perfect $k$. 

Though  these continuity arguments are rather simple, the author believes that the results described above 
are quite useful.
To illustrate their utility, we (easily) deduce a certain generalization of Theorem 2.2.1 of \cite{asok}.\footnote{The formulations of ibid. along with  that of \cite{bachcons} 
 indicate that that the "continuity" arguments applied in this paper are new to a significant part of specialists in the motivic homotopy theory.}  So, we extend this theorem (in  Proposition \ref{pasok}) to the case of an arbitrary non-orderable perfect base field (and to a not necessarily proper 
 $X/k$). Moreover, (for the sake of generality; in this proposition as well as in the central  results of this paper) we actually consider  the $\lam$-linear versions of the statements described above, where $\lam$ is an arbitrary (unital) coefficient subring of $\q$. For a smooth proper $X$ this corresponds to studying  the conditions ensuring that $X$ contains a $0$-cycle whose degree is invertible in $\lam$ and that the kernel of the degree homomorphism
 $\operatorname{Chow}_{0}(X_L)\to \z$ is killed by $-\otimes_\z \lam$
for any field extension $L/k$ (see Remark \ref{rasok}(2)); so, the case $\lam\neq \z$ may be quite interesting as well.

This   $\lam$-linear setting has some more advantages. In particular, we describe an argument deducing 
our central Theorem \ref{tbach}(i) 
 from the "slice-convergence" results of \cite{levconv} (avoiding the usage of the more complicated results of \cite{bachcons}); yet for this argument we have to assume that the characteristic $p$ of $k$ is invertible in $\lam$ whenever it is positive. However, it appears that the most interesting cases are $\lam=\z$ and $\lam=\zop$ (for $p>0$).

Another application of our results (that  
 generalizes one more statement formulated by Bachmann and requires a coefficient ring 
  containing 
	 $\frac{1}{p}$ if $p>0$) is the following one: the 
cone of  the "structure morphism" $\sinftlam(A_+)\to S^0_{\lam}$ is $\otimes$-invertible in $\shtlam$ whenever $k$ is non-orderable, $A$ is the (affine)  zero set of $\phi-a$ for $\phi$ being a non-zero quadratic form, $0\neq a\in k$, $p$ is distinct from $ 2$ and is invertible in $\lam$ if it is positive. We deduce this statement from Theorem 33 of \cite{bachinv} (whose proof is based on the usage of the Chow weight structure on $\dmck[\frac 1 p]\subset \dmk[\frac 1 p]$).

Now we describe the most original result of this paper (at least,  it appears not to be formulated in the literature 
in any form). We prove that an object $E$ of $\shlamc(k)$ belongs to $\shelam(k)(r)$ 
(to the $r$th level of the $\lam$-linearized  version of the Voevodsky's  slice filtration; we also say that the objects of $\shelam(k)(r)$ are $r$-effective) if and only if $\mklam(E)$ belongs to  $\dmelam(k)(r)$. Moreover, we establish a certain "$\thom$-connective" version of this statement. Assuming that $p$ is invertible in $\lam$ whenever it is positive, we immediately deduce the "infinite effectivity" 
 of objects in the compact motivization kernel, and  also say when  $M_{k,\lam}(E)\in \dmlam(k)_{\tdmlam \ge r+1}$ for $E\in  \shlam(k)_{\tshlam\ge r}$. 

We also note that 
continuity arguments 
similar to ones applied in this paper can be useful for the study of a wide range of motivic questions. For this reason we discuss these continuity matters in a rather detailed and "axiomatic" manner.

Let us  now describe the contents  of the paper. Some more information of this sort can be found at the beginnings of sections.

In \S\ref{sprelim} we recall some basics on (general) triangulated categories, $\sht(-)$ and $\dm(-)$,  on the cohomological dimension of fields and their Grothendieck-Witt rings of quadratic forms. We also introduce the $\lam$-linear versions of  $\sht(-)$ and $\dm(-)$ and discuss certain continuity arguments.

In \S\ref{sextres} we recall some more results on motivic categories (we formulate them in the $\lam$-linear setting). They enable us to generalize 
 certain results of Bachmann (as well as  Theorem 2.2.1 of \cite{asok}) to the case of arbitrary non-orderable base fields. We also  prove that the restriction of 
 $\mklamc$ to $2$-torsion objects is conservative over any perfect $k$. 

 \S\ref{skernel} 
we prove that the compact motivization functor $\mklamc$ "strictly respects" the slice filtrations (on $\shlamc(k)$ and $\dmlamc(k)$, respectively) as well as the (more precise) homotopy $t$-structure analogue of this result. These statements give an alternative method 
 for proving 
Theorem \ref{tbach}(i) (that is the central result of this paper); so we sketch an argument deducing it from the results of \cite{levconv} (under the additional assumption that $p$ is invertible in $\lam$ whenever it is positive; note  that 
 \cite{bachcons} relies on the results of M. Levine as well). 
 Lastly, we explain that in all our results the categories $\dmlam(-)$ may be replaced by the categories $\shmglam(-)$ of "cobordism-modules".

The author is deeply grateful to  prof. Alexey Ananyevskiy, prof. Tom Bachmann, and to the referee for their really interesting comments.

\section{Preliminaries} 
\label{sprelim}

In \S\ref{snotata} we introduce some notation and a few conventions that we will use throughout the paper. 

In \S\ref{scomploc} we discuss compactly generated triangulated categories along with their $\lam$-linear versions (for $\lam\subset \q$, i.e., we invert some set $S$ of primes in a triangulated category $\cu$ to obtain the corresponding $\cu_{\lam}$).

In \S\ref{sshcont} 
we recall some basics on the motivic categories $\sht(-)$ and $\dm(-)$. We also note that these statements generalize to $\shlam(-)$ and $\dmlam(-)$. Moreover, we describe  (abstract versions of) our basic continuity arguments.

In \S\ref{switt} we recall some 
 well-known properties of the cohomological dimension  of (essentially finitely generated) fields and relate the Grothendieck-Witt ring of $k$ to $\shtk(S^0,S^0)$.

\subsection{Some notation and terminology}\label{snotata} 

\begin{itemize}

\item
For categories $C,D$ we write 
$D\subset C$ if $D$ is a full 
subcategory of $C$.

 For a category $C$ and $ X,Y\in\obj C$, the set of  $C$-morphisms from  $X$ to $Y$ will be denoted by
$C(X,Y)$.

\item
Below $\cu$  will always denote a triangulated category.

For $E\in \obj \cu$ we will say that it is {\it $2$-torsion} if there exists $t>0$ such that $2^t\id E=0$. 

We will use the term {\it exact functor} for a functor of triangulated categories (i.e., for a  functor that preserves the structures of triangulated
categories).

\item
For a triangulated category $\cu$ and some $D\subset \obj \cu$ we will call the smallest subclass $D'$ of $\obj \cu$ that contains $D$ and is closed with respect to all $\cu$-extensions and retractions the {\it envelope} of $D$ (so, it is thick if $D[1]= D$).

\item
Below  $k$ and $F$ will always be  perfect fields of characteristic $p$ (and the case $p=0$ will be the most interesting for us); $k$ will usually denote "the base field" for the motivic categories that we consider at the given moment (whereas $F$ will often run through all perfect fields). 
$L$ will  denote a field of characteristic $p$ also; 
we will not assume $L$ to be perfect (by default). 

 The category of all perfect fields  will be denoted by $\pfi$. 

\item
When writing $k=\inli k_i$ we will always assume that $k_i$ 
 form a directed system of perfect fields (so, this is an inductive limit). 

\item
Now let $\ff$ be a $2$-functor  from $\pfi$ into a certain $2$-category of categories (that will actually  be the $2$-category of tensor triangulated categories for all the examples of this paper). 
Then for $m:k\to k'$ being a $\pfi$-morphism and $E\in \obj \ff(k)$  the object $\ff(m)(E)$ of $ \ff(k')$ will often be denoted by $E_{k'}$. If an object  $E'$ of $\ff(k')$  is isomorphic to $ E_{k'}$ (for some $E\in \obj \ff(k)$) then we will say that $E'$ is {\it defined over $k$}.

\item
 We will say that the {\it continuity property for morphisms} is fulfilled for $\ff$ 
 if 
$\ff(k)(M^0_k,N^0_k)\cong \inli _i \ff(k_i)(M^0_{k_i},N^0_{k_i})$ whenever $k=\inli_i k_i$, all these fields are extensions of a certain perfect field $k_0$, whereas $M^0$ and $N^0$ are some objects of $\ff(k_0)$. 

This assumption is (an important) part of the following   continuity property for $\ff$ (cf. \S4.3 of \cite{cd}): we will say that $\ff$ is {\it continuous} if we have $\ff(k)\cong \inli \ff(k_i)$ whenever $k=\inli k_i$ (i.e., we consider the $2$-category colimit with   the transition functors being the result of applying $\ff$ to the corresponding $\pfi$-morphisms).

\item
We will say that $k$ is non-orderable whenever $-1$ is a sum of squares in it.

\item
 $\sv$ will denote the 
 set of (not necessarily connected) smooth $k$-varieties (and in some occasions we will consider $\sv$ as a category). More generally, $\sv(F)$ will denote the set of smooth $F$-varieties. 

$\pt$ will always denote the point $\spe k$ (over $k$); $\p^1$ will denote the projective line $\p^1(k)$, and $\afo=\afo(k)$ is the affine line.
\end{itemize}

\subsection{On compactly generated categories and localizing coefficients for them}\label{scomploc} 

In this subsection  $\cu$ will denote a triangulated category  closed with respect to all small coproducts. 
We recall the following (more or less) well-known definitions.

\begin{defi}\label{dcomp}
1. We will say that an object $M$ of $ \cu$ is {\it compact} whenever the functor $\cu(M,-)$ respects coproducts.

2. We will say that a class $C=\{C_i\}\subset \obj \cu$ generates a subcategory $\du\subset \cu $ {\it as a localizing subcategory} if $\du$ equals 
 the smallest full strict triangulated subcategory of $\cu$ that is closed with respect to small coproducts and contains $C$.


3. We will say that $C=\{C_i\}$  {\it compactly generates} $\cu$ (or that the $C_i$ compactly generate $\cu$) if $C$ is a set, all $C_i$ are compact (in $\cu$), and  $C$ generates $\cu$  as its own localizing subcategory.

We will say that $\cu$ is compactly generated whenever there exists some set of compact generators of this sort.

\end{defi}

\begin{rema}\label{rcomp}
 Recall 
(see  Lemma 4.4.5 of \cite{neebook}) that if
$C$ compactly generates $\cu$ then the full subcategory $\cu^c$ of compact objects of $\cu$ is the smallest thick 
subcategory of $\cu$ containing $C$ (i.e., if $\obj \cu^c$ is the envelope of $\cup_{j\in \z} C[j]$ in the sense described in \S\ref{snotata}).
Moreover, 
 $\cu^c$ is {\it idempotent complete}, i.e., any idempotent endomorphism gives a splitting in it. 
\end{rema}

In the current paper we use the "homological convention" for $t$-structures (following \cite{morintao} and \cite{bachcons}). 
Thus a $t$-structure $t$ for $\cu$ gives homological functors $H^t_j$ from $\cu$ to the heart $\hrt$ of  $t$ such that $H^t_j=H_0^t\circ [-j]$ for any $j\in \z$.
If $t$ is {\it non-degenerate} (i.e., the collection $\{H^t_j\}$ for $j\in \z$ is conservative; we will call these functors $t$-homology) then $E\in \cu_{t\le 0}$ (resp. $E\in \cu_{t\ge 0}$) if and only if $H_j^t(E)=0$ for all $j>0$ (resp. $j<0$).

We recall the following existence statement.

 \begin{pr}\label{paisle}
Let $C\subset \obj \cu$ be a set of compact objects. Then there exists a unique  $t$-structure $t$ for $\cu$ such that $\cu_{t\ge 0}$ is the smallest subclass of $\obj \cu$ that contains  $C$ and  is stable with respect to extensions, the  suspension $[+1]$, and arbitrary (small) coproducts.
\end{pr}
\begin{proof} This is Theorem A.1 of \cite{talosa}.

\end{proof}
\begin{rema}\label{rtgen}
1. Recall that $E\in \obj \cu$ determines its $t$-decomposition triangle $E_{t\ge 0}\to E \to E_{t\le -1}$ (with $E_{t\ge 0}\in \cu_{t\ge 0}$ and $E_{t\le -1}\in \cu_{t\le -1}=\cu_{t\le 0}[-1]$) in a functorial way.

Moreover, for the $t$-structure given by Proposition \ref{paisle} the corresponding functors $-_{t\ge 0}$ and  $-_{t\le -1}$ respects $\cu$-coproducts; see Proposition A.2 of ibid.

2. Under the assumptions of the proposition we will say that the $t$-structure $t$ is generated by $C$.

3. If $C$ is suspension-stable (i.e., $\{C_i\}[1]=\{C_i\}$) then the classes $\cu_{t\ge 0}$ and $\cu_{t\le 0}$ are suspension-stable as well. 
Thus  $\cu_{t\ge 0}$ is the class of objects of the localizing subcategory $\du$ generated by $\{C_i\}$, and $E\mapsto E_{t\ge 0}$ yields the right 
 adjoint to the embedding $\du\to \cu$. This functor is clearly exact; this setting is called the {\it Bousfield localization} one (in \cite{neebook}). 
\end{rema}

 We also recall some basics on "localizing coefficients" in a triangulated category.

Below 
$S\subset \z$ will always be a set of prime numbers; the ring  $\z[S\ob]$ will be denoted by $\lam$.
We will 
often assume that $S$ contains $p$ whenever $p$ is the characteristic of our base field $k$ 
 and $p>0$. 

\begin{pr}\label{plocoeff}

 Assume that $\cu$ is compactly generated by a small subcategory $\cu'$.
Denote
 by $\cu_{S-tors}$ the localizing subcategory of $\cu$ 
(compactly) generated by $\co(c'\stackrel{\times s}{\to} c')$ for $c'\in \obj \cu',\ s\in S$.


Then the following statements are valid.

\begin{enumerate}
\item\label{icg1}  $\cu_{S-tors}$ 
also contains all  cones of $c\stackrel{\times s}{\to} c$ for $c\in \obj \cu$ and $ s\in S$.



\item\label{icg2} The Verdier quotient category $\cu_\lam
=\cu/\cu_{S-tors}$ exists (i.e., the morphism groups of the localization are sets);
the localization
functor $l\colon\cu\to \cu_\lam$ respects all coproducts and converts
compact objects into compact ones.  Moreover, $\cu_\lam$ 
is generated by $l(\obj \cu')$ as a localizing subcategory. 

\item\label{icg3} For any $c\in \obj \cu$, $c'\in \obj \cu'$, we have $\cu_\lam(l(c'),l(c))\cong \cu(c',c)\otimes_\z \lam$.


\item\label{icg5} $l$ possesses a right adjoint $G$ that is a full embedding functor. 
The essential image of $G$  consists  of those $M\in \obj \cu$ 
 such that $s\cdot\id_M$ is an automorphism for any $s\in S$ (i.e, $G(\cu)$ is essentially the maximal full $\lam$-linear subcategory of $\cu$).

\item\label{icgdi} Assume that $\du$ is also a compactly generated  (triangulated) category; define $\du_\lam$, $l_{\du}$ and $G_{\du}$ as the $\du$-versions of $\cu_{\lam}$, $l$, and $G$, respectively. Then  any functor $F:\cu\to \du$ that respects coproducts can be 
 canonically completed to a diagram
$$\begin{CD}
 \cu@>{l}>>\cu_{\lam} @>{G}>>\cu\\
@VV{F}V@VV{F_{\lam}}V@VV{F}V \\
\du@>{l_{\du}}>>\du_{\lam} @>{G_{\du}}>>\du
\end{CD}$$
where $F_{\lam}$ is a certain exact functor respecting coproducts. 

\end{enumerate}
\end{pr}
\begin{proof}
See Proposition 5.6.2(I) of \cite{bpgws} (cf. also Proposition A.2.8 and Corollary A.2.13 of \cite{kellyth} along with  Appendix B of \cite{levconv}).
\end{proof}

\begin{rema}\label{rnot} 
For $S=\{l\}$ (i.e., consisting of a single prime) we will write  $\cu[l\ob]$ instead of $\cu_{\z[\frac 1 l]}$. 

Moreover, for a triangulated category $\cu$ that is a value of a $2$-functor $\du$ from $\pfi$ (i.e., if $\cu=\du(F)$ for some perfect field $F$)   its $\lam$-linear version will be denoted by $\du_{\lam}(F)$ instead of $\du(F)_{\lam}$. 
\end{rema}

\subsection{On motivic categories and continuity}\label{sshcont}

Now we recall some basics properties of  triangulated motivic categories (that were defined by Voevodsky and Morel). For our purposes it will be sufficient to consider them over perfect fields only; yet note that a much more general theory is currently available (thanks to the works of Ayoub, Cisinski, and D\'eglise). Respectively, instead of morphisms of base schemes we will consider morphisms of fields. 
 The tensor product operations on 
 our categories will be denoted by $\otimes$.


\begin{pr}\label{pcont}
\begin{enumerate}
\item\label{ishdm1} 
There exist covariant $2$-functors $k\mapsto \shtk$ and $k\mapsto \dmk$ 
 (see \S4.2 of \cite{degmod} for the latter) from $\pfi$ into the $2$-category of tensor triangulated categories.

The categories $\shtk$ and $\dmk$ are closed with respect to arbitrary small coproducts and the tensor products for them respect coproducts (when one of the arguments is fixed). 
Moreover, for  a morphism $m:k\to k'$ the functors $\msh$ and $\mdm$ also respect all coproducts and the compactness of objects.

\item\label{ishdmsinf} There exist  functors $\sv\to \shtk:\ X\mapsto \sinft (X_+)$ (one may consider this as a notation) and $\mg:\sv\to \dmk$; they factor through the corresponding subcategories of compact objects $\shtck$ and $\dmck$, respectively.
Moreover, these two functors 
 convert the products in $\sv$ into the tensor products in $\shtk$ and $\dmk$, respectively,   and  convert the projection $\afo\to \pt$ into isomorphisms.

\item\label{ishdmmg}
For any  $k$ there is an exact tensor functor $M_{k}:\shtk\to \dmk$ (the {\it motivization} functor) that respects coproducts and the compactness of objects; we have $\mk(\sinft(X_+))\cong \mg(X)$ for any $X\in \sv$. 
Moreover, for  any $\pfi$-morphism $m:k\to k'$ the diagram
$$\begin{CD}
\sv(k) @>{\sinft(-_+)}>> \shtk @>{M_k}>>\dmk\\
@VV{X\mapsto X_{k'}}V @VV{\sht(m)}V@VV{\dm(m)}V \\
\sv(k') @>{\sinft{}_{k'}(-_+)}>> \shtkp@>{M_{k'}}>>\dmkp
\end{CD}$$ 
 is commutative. 

\item\label{ishdmadj} $M_k$ possesses a right adjoint $U_k$ that respects coproducts. Furthermore, the functor $U_k\circ M_k$ is isomorphic to $-\otimes \emz$ for a certain object $\emz$ of $ \shtk$. 

\item\label{ishdmun} 
The objects $S^0=\sinft(\pt_+)$ and $\z=\mg(\pt)$ (we omit $k$ in this notation) are tensor units of the corresponding motivic categories, 
and we have $\dmk(\z,\z)\cong \z$. 


\item\label{ishdmtw} 
Denote by $T$ the complement to $\sinft(\pt_+)$ in $\sinft(\p^1_+)$ (with respect to the natural splitting), and denote by $\z(1)[2]$ the complement to $\mg(\pt)$ in $\mg(\p^1)$.
Then these objects are $\otimes$-invertible in the corresponding categories, and $M_k(T)\cong \z(1)[2]$.

 The $i$th iterates of the functors $-\otimes (T[-1])$ and $-\otimes (\z(1)[1])$ will (abusively) be denoted by $-\{ i\}$ for all $i\in \z$. 
\footnote{Note here that  the "usual" Tate twist $-(1)$ (in the convention introduced by Voevodsky) clearly equals $-\{1\}\circ [-1]$.}

\item\label{ishdmcomp} The category $\shtk$ (resp. $\dmk$) is compactly generated (see Definition \ref{dcomp}) by the objects $\sinft(X_+)\{ i \}$ (resp. $\mg(X)\{ i \}$) for $X\in \sv, i\in \z$.

\item\label{ishdmeps}  For any $k$ there exists a canonical idempotent $\shtk$-endomorphism $\epsilon$ of $S^0$ (see \S6.1 of \cite{morintao}) such that  $M_k(\epsilon)=-\id_{\z}$. 

 \item\label{ishdmcont} The 
$2$-functors $\shtc(-)$ and  $\dmc(-)$  are continuous in the sense described  
 in \S\ref{snotata} 
(i.e.,   
   $\shtc(k)\cong \inli_i \shtc(k_i)$ and $\dmc(k)\cong \inli_i \dmc(k_i)$ whenever  $k=\inli k_i$).  


\end{enumerate}
\end{pr}
\begin{proof}
All of these assertions are rather well-known except possibly the first part 
 of the last one, that can be found in  Example 2.6(1) of \cite{cdint} (see also \S6.1 and Remark  6.3.5 of \cite{morintao} 
 for assertion \ref{ishdmeps}). 
\end{proof}

Now we introduce the $\lam$-linear versions of our motivic triangulated categories. 
As we will briefly explain, these categories are easily seen to fulfil the natural analogues of the properties listed in Proposition \ref{pcont}.  
For the convenience of the readers we note that the following 
 proposition is not necessary for the understanding of \S\ref{seasy}. 

\begin{pr}\label{prcont}
 Choose a set of primes $S$, set $\lam=\z[S\ob]$, and consider the 
$2$-functors $\shtlam$ and $\dmlam$ from $\pfi$ into the $2$-category of tensor triangulated categories (see Proposition \ref{plocoeff}(\ref{icg2},\ref{icgdi}) and the convention introduced in Remark \ref{rnot}).

Then the following statements are valid.

\begin{enumerate}
\item\label{iel1}
 The functors of the type $\shlam(m)$  and $\dmlam(m)$ (where $m$ is a morphism of perfect fields) respect the  compactness of objects and all coproducts. 

Moreover, the tensor products in these categories respect coproducts (when one of  the arguments is fixed).

\item\label{iel2}
The natural $\lam$-linear versions 
 of Proposition \ref{pcont}(\ref{ishdmsinf}--\ref{ishdmcont}) 
are also valid.

\item\label{iel3}
The functors  $F\mapsto \ns\subset \obj \dmlamc(F)$ and $F\mapsto \{\mglam(\sv(F))\{r\}[j]\}$ (for $F$ being a perfect field and any fixed $r,j\in\z$) are 
 $\dmlamc$-continuous.
\end{enumerate}

\end{pr}
\begin{proof}
These statements 
 mostly easily follow from Proposition \ref{pcont} combined with Proposition \ref{plocoeff}. However, one should also invoke Remark \ref{rcomp} to obtain that the functors in question respect the compactness of objects 
 along with the $\lam$-linear version of Proposition \ref{pcont}(\ref{ishdmcont}) and assertion \ref{iel3}.  

\end{proof}


\begin{rema}\label{rcont}

We will now discuss some more notation and properties for the $2$-functors $\shlamc$ and $\dmlamc$; certainly, they can  be applied in the case $\lam=\z$ (i.e., for $S=\emptyset$).

1. The restriction of $\mklam$ to the subcategory $\shlamc(k)$ of compact objects (with its image being the corresponding $\dmlamc(k)$) will be denoted by $\mklam^c$.

2. We will need a certain property of continuity for families of subsets of $\obj\dmlamc(-)$. 
To avoid  (minor) set-theoretical difficulties, 
till the end of the section will assume that  $\dmlamc(F)$ is a small category for any perfect field $F$.
This technical assumption is easily seen 
not to affect the results below (and we may actually adopt it in the rest of this paper as well).

So, 
let $O$ be a subfunctor of the functor $\obj \dmlamc$ from $\pfi$ to the category of sets
(i.e., $O(F)\subset \obj \dmlamc(F)$ for all perfect $F$, and $O(m)$ for a morphism $m:k\to k'$ of perfect fields is given by the restriction of $\dmlam(m)$ to $O(k)$). Then we will say that $O$ is {\it $\dmlamc$-continuous} if it satisfies the following condition:
 for  $k=\inli k_i$ 
and any $M\in  O(k)$ 
there exists some $k_0\in \{k_i\}$ and $M^0\in O(k_0)$ such that $M\cong M^0_{k}$ (i.e.,  $M\cong O(m_0)(M^0)$  for the corresponding $m_0:k_0\to k$; see  \S\ref{snotata}). Note that latter condition is clearly equivalent to the set of $\dmlamc(k)$-isomorphism classes in $O(k)$ 
 to be the direct limit  of the sets of  $\dmlamc(k_i)$ -isomorphism classes in   $O(k_i)$.

3. Below we will apply the following consequence of continuity: for any $\dmlamc$-continuous $O$ the property that for some $E\in \obj \shlamc(k)$, the object $M^c_{k,\lam}(E)\in \obj \shlamc(k)$  belongs to $O(k)$ is "continuous" as well. This means the following: 
 if $k=\inli k_i$,  $E\in \obj\shlamc(k)$, and $\mklam(E)\in O(k)$, then there exists some $k_j\in \{k_i\}$ along with   $E^j\in \obj\shlamc(k_j)$ such that $M_{k_j,\lam}(E^j)\in O(k_j)$ and 
$E^j_k\cong E$ (see \S\ref{snotata}). 

Indeed, 
 the continuity property for $\shlamc(-)$  allows us to choose some $k_0\in \{k_i\}$ such that $E$ is defined over it (i.e., such that there exists $E^0\in \obj \shlamc(k_0)$ with
$E\cong 
E^0_k$).  Next, the $\dmlamc$-continuity of $O$ gives the existence of $k_1\in \{k_i\}$  and $M^1\in O(k_1)$ such that $
M^1_k \cong \mklam(E)$. Furthermore, the 
 continuity property for morphisms in $\dmlamc(-)$ (see \S\ref{snotata})  gives the existence of $k_2\in  \{k_i\}$ that contains both $k_0$ and $k_1$ such that 
$(M_{k_0,\lam}(E^0)) _{k_2}\cong M^1_{k_2}$. Thus we can take  
$k_j=k_2$, $E^j=E^0_{k_2}$ (since  $M_{k_2,\lam}(E^0_{k_2})\cong M^1_{k_2}\in (O(k_1))_{k_2}\subset O(k_2)$).

4. Now we describe some "tools" for constructing $\dmlamc$-continuous functors; we will apply them along with Proposition \ref{prcont}(\ref{iel3}).

Firstly, 
the functors  $F\mapsto \ns\subset \obj \dmlamc(F)$ and $F\mapsto \{\mglam(\sv(F))\{r\}[j]\}$ (for $F$ being a perfect field and any fixed $r,j\in\z$) are obviously $\dmlamc$-continuous.

Next, the "union" of any set of continuous functors is easily seen to be continuous.

Lastly, if $O$ is $\dmlamc$-continuous then the functor sending $F$ into 
the 
 envelope of $O(F)$ (in $\dmlamc(F)$) is $\dmlamc$-continuous as well (recall that we assume $\dmlamc(F)$ to be small). 


\end{rema}

\subsection{On  cohomological dimensions and Grothendieck-Witt rings}\label{switt}

As we have said in \S\ref{snotata}, $L$ always denotes some  (not necessarily perfect)  characteristic $p$ field. 
We recall the following well-known facts.

\begin{pr}\label{pcohdim}
Let $L$ be a finitely generated field (i.e., $L$ is finitely generated over its prime subfield). 
Then the following statements are valid.


1. If $L$ is  non-orderable then its  cohomological dimension  (at any prime) 
is finite.

2. 
The cohomological dimension of $L$ (whether it is finite or not)  equals the  one of the perfect closure $L^{perf}$ of $L$.

\end{pr}
\begin{proof}
1. See \cite{serregalois}, \S  II.3.3 and  II.4.2.

2. It suffices to note that the absolute Galois group of $L$ equals the one of its perfect closure.


\end{proof}

The following easy lemma follows immediately.

\begin{coro}\label{cohdim}
If $k$ is non-orderable then it may be presented as a filtered direct limit of perfect fields of finite cohomological dimension. 

\end{coro}
\begin{proof} 
It suffices to present $k$ (recall that we assume it to be perfect) as the direct limit of the perfect closures of its finitely generated subfields, and apply the previous proposition.

\end{proof}

\begin{rema}\label{rvcd}
1. Note however that below (everywhere except in \S\ref{salt}) it will actually be sufficient to present $k$ as the direct limit of fields of finite cohomological $2$-dimension.

2. Recall that the  virtual cohomological $2$-dimension of a field $L$ of characteristic $\neq 2$ may be defined as the  cohomological $2$-dimension of $L[\sqrt{-1}]$.
Thus any finitely generated field of characteristic $\neq 2$ is of finite virtual  cohomological $2$-dimension.

\end{rema}

Now we recall some basics on Grothendieck-Witt rings and their relation to $\sht(-)$.

\begin{rema}\label{rwitt}
1. As shown in \S6.3 of \cite{morintao} (see Theorem 6.3.3 and Lemma 6.3.8 of ibid.), $\shtk(S^0,S^0)\cong  GW(k)$ (the  Grothendieck-Witt ring of $k$). If $p\neq 2$ then the latter is the  Grothendieck group of  non-degenerate $k$-quadratic forms. It
is isomorphic to the kernel of $W(k)\bigoplus \z\to \z/2\z$, where $W(k)$ is the Witt ring of (quadratic forms over $k$) and the projection $W(k)\to \z/2\z$ is given by the parity of the dimension of quadratic forms. In the case  $p=2$ one should consider symmetric bilinear forms instead of quadratic ones here.

 As mentioned in the beginning of \S2 of  \cite{arael}, if $p\neq 2$ then the group $W(k)$ is an extension of the free abelian group whose generators correspond to orderings on $k$ by a torsion group. Thus the kernel of $M_{k*}:\shtk(S^0,S^0)\to \dmk(\z,\z)$ is torsion if and only if $k$ is non-orderable (at least, in the case $p\neq 2$; note that below we will apply this statement in the case $p=0$ only).

2. It is no wonder that structural results on Witt rings of fields play a very important role in motivic homotopy theory. In particular, they were crucial for \cite{bachcons}, \cite{levconv}, and \cite{asok}. Information of this sort was also actively used in the previous version of the current paper; yet the corresponding arguments were essentially incorporated in the current version of \cite{bachcons}  (resulting in Lemma 19 of ibid.). 

\end{rema}

\section{
Main conservativity of motivization results 
}\label{sextres}

The main result of this section is that (the "compact version" (b) of) Theorem 16 of \cite{bachcons} 
 can be extended to the case when $k$ is an arbitrary non-orderable field. 
 Moreover, the restriction of  $\mk^c$ to $2$-torsion objects is conservative for any $k$. 

So, in \S\ref{seasy} we prove the "triangulated parts" of these results. We deduce them from 
similar results of ibid. (where certain cohomological dimension finiteness was assumed) using a simple continuity argument (that is a particular case of the reasoning described in Remark \ref{rcont}(3)). 
We also note that the conservativity of $\mk^c$ never extends to 
"the whole" $\mk$; 
 moreover, $\mk^c$ is never conservative if $k$ is not non-orderable (i.e., if it is formally real).

In \S\ref{sthom} we  study the homotopy $t$-structures 
and  (Voevodsky's) slice filtrations  for 
 $\shtlam(-)$ and $\dmlam(-)$ (for a coefficient ring $\lam\subset \q$);  their properties follow from their 
 well-known $\z$-linear versions. 

In \S\ref{sdiff}  we prove the $\lam$-linear version of (the stronger part of)  
Bachmann's theorem  over an arbitrary non-orderable $k$, stating that  the $m$-homotopy connectivity of $M_{k,\lam}(E)$ for $E\in \obj \shlamc(k)$ ensures the $m$-homotopy connectivity (with respect to the homotopy $t$-structure $\tshlam$) of $E$  itself. We also give the following immediate applications of our results (for $k$ being any non-orderable perfect field): we prove the corresponding generalization of Theorem 2.2.1 of  \cite{asok}, and prove that Theorem 33 of \cite{bachinv} (on the $\otimes$-invertibility of certain motives of affine quadrics) may be carried over to motivic spectra.

\subsection{On "$\z$-linear  triangulated conservativity"} 
 \label{seasy}


Now we prove the weaker versions of our conservativity results. 

\begin{theo}\label{peasy}
I. Assume that $k$ is a non-orderable field. Then the following statements are valid.

\begin{enumerate}
\item\label{ieps}
There exists 
 $N\ge 0$ such that $2^N(1+\epsilon)=0$ in $\sht(k)$ (see Proposition \ref{pcont}(\ref{ishdmeps})) and $2^N\eta=0$, where $\eta$ is the (Morel's) 
  stable algebraic Hopf map $S^0\{1\}\to S^0$.

\item\label{ibacht25} The restriction $M_k^c:\shtck\to \dmck$
 of the motivization functor $M_k$ to  compact objects 
  is conservative.

II. Let $E$ be a $2$-torsion   (see \S\ref{snotata}) object of $\shtck$, where $k$ is an arbitrary perfect field. Then $E=0$ whenever $M_k^c(E)=0$.
\end{enumerate}
\end{theo}

\begin{proof}
I.1. By Lemma 6.7 of \cite{levconv}, the assertion is fulfilled if $p>0$. Thus we can assume $p\neq 2$.
  
Now, $1+\epsilon$ belongs to the image in $\shtk(S^0,S^0)\cong  GW(k)$ of the class $[x^2]-[-x^2]$; 
see Remark \ref{rwitt}(1). 
Hence the first part of the assertion 
 easily follows from 
 Proposition \ref{pcohdim}. The second part of the assertion follows from the first one immediately
by Lemma 6.2.3 of \cite{morintao}. 

2. According to Theorem 16 of \cite{bachcons} (see version (b) of the first part of the theorem), the statement is valid if the  cohomological $2$-dimension of $k$ is finite. 

Next, in the general case 
 Corollary \ref{cohdim} enables us to present $k$ as $\inli k_i$ (recall here the conventions described in \S\ref{snotata}) so that the  cohomological $2$-dimensions of $k_i$ are finite. Thus to finish the proof it suffices to recall that the correspondence $F\mapsto \ns\subset \obj \dmc(F)$ is $\dmc$-continuous (see Remark \ref{rcont}(4)); here we take $\lam=\z$ and apply part 3 of this remark).

Now we explain this continuity argument in our concrete situation (for the sake of those readers that have problems with Remark \ref{rcont}).

Assume that $\mk(E)=0$ for some $E\in \obj \dmck$. 
By the continuity property for $\shtc(-)$ (see Proposition \ref{pcont}(\ref{ishdmcont})) there exists 
$k_0\in \{k_i\}$ such that  $E$ is defined over $k_0$ (i.e., there exists  $E^0\in \obj\shtc(k_0)$ such that 
$E^0_k\cong E$; cf. \S\ref{snotata}).  Next, the continuity property  for the morphisms in $\dmc(-)$ (see \S\ref{snotata})  gives the existence of  $k_1\in \{k_i\}$ such that $k_1$ is an extension of $k_0$ and 
$M_{k_1}(E^0_{k_1})=0$. Hence applying Theorem 16 of \cite{bachcons} to 
$E^0_{k_1}$  we obtain $E^0_{k_1}=0$. Thus the object $E\cong (E^0_{k_1})_k$ is zero as well. 

II. The proof is rather similar to that of assertion I.2. 
Firstly, that assertion enables us to assume that $k$ is   formally real; so, we restrict ourselves to the case $p=0$.

Then  $k=\inli k_i$, where $k_i$ are  finitely generated extensions of $\q$. Similarly to the previous proof, the   continuity property for $\shtc(-)$ gives the existence of 
$k_0\in \{k_i\}$ and $E^0\in \obj\shtc(k_0)$ such that 
$E^0_k\cong E$. Moreover, the continuity property for morphisms in  $\shtc(-)$ enables us to assume that $E_0$ is $2$-torsion. 

Thus it suffices to prove our assertion for $k$ being an orderable finitely generated field. Hence it remains to apply Lemma 19 of ibid. (along with Remark \ref{rvcd}(2)). 
\end{proof}

\begin{rema}\label{rshmi}

Now we give some examples demonstrating that the assumptions of our 
 theorem are necessary. 

1.  
The "whole" 
$M_k$ is not conservative for any (perfect) $k$. 
Indeed,  consider the homotopy colimit $S^0[\eta\ob]$ of the sequence of morphisms  $S^0\stackrel{\eta\{-1\}}{\to}S^0 \{-1\}\stackrel{\eta\{-2\}}{\to}S^0 \{-2\}\to \dots$ (originally considered in Definition 2 of \cite{anlepa}; note yet 
  that the definition of $\eta$ in ibid. differs from our one by $-\{1\}$).
Since $M_k(\eta)=0$,  we have $M_k(S^0[\eta\ob])=0$ (see Lemma 1.6.7 of \cite{neebook}).  On the other hand, Theorem 1 of \cite{anlepa} easily implies that $S^0[\eta\ob]\neq 0$.

Now let us assume that $k$ is non-orderable. Then part I.\ref{ieps} of our 
 theorem easily implies that $S^0[\eta\ob]$ is a $2$-torsion object (cf. part II of the 
theorem). Moreover, we obtain that the kernel of ("the whole") $M_k$ is not generated by 
 the one of $M_k^c$ (as a localizing subcategory of $\sht(k)$) in this case.

2. Furthermore,  $\shtohk$ may be considered as a subcategory of $\shtk$ (see Proposition \ref{plocoeff}(\ref{icg5})).
Next, recall that $\shtohk$  naturally splits as the product of certain triangulated categories $\shtplk$ and $\shtmik$; see the text preceding Lemma 6.7 of \cite{levconv}.  Moreover, the objects of $ \shtmik$ inside $\shtohk$ are characterized by the condition $\eps =-\id$ (see Proposition \ref{pcont}).
Since the functor $M_k$ kills  $1+\eps$ (see part \ref{ishdmeps} of the proposition), we obtain that it annihilates $\shtmik$. 

Now, if $k$ is formally real (i.e., not non-orderable) then the image of $S^{0}(k)$ of $ \shtmik$ is not torsion. Indeed, recall that $\shtplk_\q\cong \dmk_\q$ (see Theorem 16.2.13 of \cite{cd}) 
whereas   $\shtk(S^0,S^0)\otimes \q\not \cong \dmk(\z,\z)\otimes \q$ in this case (see Remark \ref{rwitt}(1)).

On the other hand, $\shtplk=\shtohk$ if $-1$ is a sum of squares in $k$, i.e., if $k$ is unorderable. Indeed, in the case $\cha k>0$ this fact is given by Lemma 6.8 of \cite{levconv}; for $\cha k=0$ this statement can be easily extracted from the proof of Lemma 6.7 of ibid. Yet the author does not know whether the kernel of $M_k$ can consist of torsion objects only in this case  (here the answer may certainly depend on $k$). 


3. If $-1$ is a not a sum of squares in $k$ (i.e., $k$ is formally real) then the kernel of $\mk^c$ is non-zero as well. Indeed, the object $C=\co(2\id_{S^{0}(k)}+\epsilon)$ is clearly compact, and the long exact sequence $\dots \to \shtk(S^0,S^0)\cong GW(k)\stackrel{ \times(2[x^2]-[-x^2])}{\to } \shtk(S^0,C)\to \shtk(S^0,S^0[1])=\ns$ (see Remark \ref{rwitt}(1)) easily implies that $C\neq 0$ (since considering the split surjection of $GW(k)$ to $ \z$ corresponding to any ordering on $k$ one obtains  $\shtk(S^0,C)\supset \z/3\z$). Yet $M_k(C)=0$ since 
$M_k(2\id_{S^{0}(k)}+\epsilon)=\id_\z$.


4. Clearly, for any $E\in \obj \shtc(k)$ a cone $E/2$ of  the morphism $E\stackrel{2\id_E}{\to} E$ is a $2$-torsion object (that is surely annihilated by $4$). Thus part II of the proposition above implies that  $M_{k}(E)$ can vanish only if (the endomorphism ring of) $E$ is uniquely $2$-divisible. So it seems reasonable to conjecture that $M^c_{k}(E)$  vanishes  
  only if $E$ is an odd torsion object.  

On the other hand, the odd torsion in the kernel of $M_{k}^c$ may be quite "large" if $k$ is formally real. In particular,  $M_{k}$ kills $C\otimes \obj \sht(k)$, where $C$ is the object constructed 
 above. 
Note that one can also easily construct $l$-torsion objects "similar to $C$" for $l$ being any odd integer.

More generally,  note that 
the elements of the kernel of $M_{k,\lam}^c$ are  uniquely $2$-divisible (i.e., are $\zoh$-linear) for any choice of $S$ (and so, of $\lam$) by Corollary \ref{ceasylam} below. We conjecture that this kernel consists of odd torsion elements only whenever $2\notin S$.

5. 
Theorem  \ref{peasy}(I.1) is certainly not quite new; cf. Remark 1.2.8(2) of \cite{degorient}.

\end{rema}

\subsection{More auxiliary results: homotopy $t$-structures,  slice filtrations, and their continuity}\label{sthom} 

As always, $S$ will denote some set of primes, $\lam=\z[S\ob]$. 
Starting from this section we will freely use the notation and results of \S\ref{sshcont}. 

\begin{defi}\label{deff}
1. Denote by $\tshlam$ (resp. $\tdmlam$) the $t$-structure on $\shtlamk$ (resp. on $\dmlamk$) generated by $\sinftlam(X_+)\{i\}$ (resp. by $\mglam(X)\{i\}$) for $X\in \sv$, $i\in\z$ (see Remark \ref{rtgen}(2)). 
We will call these $t$-structures {\it homotopy} ones.

We will say that $E\in \obj\shtlamk$ is {\it homotopy connective} if it belongs to $\shtlamk_{\tshlam\ge i}$ for some $i\in \z$.

2.  Denote by $\shelam(k)$ (resp. $\dmelam(k)$) the localizing subcategory of $\shlamk$ (resp. of $\dmlamk$) generated by  $\sinftlam(X_+)$ (resp. by $\mglam(X)$; so, we follow the convention introduced in Remark \ref{rnot}). 

Obviously, $\shelam(k)\{1\}=\shelam(k)(1)\subset \shelam(k)$ and   $\dmelam(k)\{1\}=\dmelam(k)(1)\subset \dmelam(k)$; we will call the filtration of $\shlamk$ by $\she(k)\{i\}$ (resp. of $\dmk$ by $\dmeb(k)\{i\}$) for $i\in \z$ the {\it slice} filtration. We will say that the elements of $\cap_{i\in \z}\obj\she(k)\{i\}$ 
and of $\cap_{i\in \z}\obj \dmeb(k)\{i\}$ are {\it infinitely effective}. 

We will say that $E\in \obj\shtlamk$ is {\it slice-connective}  if it belongs to $\obj\shelam\{i\}$ for some $i\in \z$.
\end{defi}

We will omit $\lam$ in this notation if $\lam=\z$. 

\begin{rema}\label{reff}
1. For any $X\in \sv$ we have $\sinftlam(X_+)\in \obj \shelam(k)\cap \shlam(k)_{\tshlam\ge 0}$ and $\mglam(X)\in \obj \dmelam(k)\cap \dmlam(k)_{\tdmlam\ge 0}$. Hence for any compact object $E$ of $\shlam(k)$ (resp. of $\dmlam(k)$) there exists $r\in \z$ such that  $E$ belongs to $\obj \shelam(k)\{r\}\cap \shlam(k)_{\tshlam\ge r}$ (resp. to $\obj \dmelam(k)\{r\}\cap \dmlam(k)_{\tdmlam\ge r}$); here we apply Remark \ref{rcomp}.

2. In \cite{bachcons} the objects that we call homotopy connective were said to be just connective.
\end{rema}

Now let us establish some more basic properties of these filtrations (and recall that the category $\shtck[\frac 1 p]$ is rigid).

\begin{pr}\label{peff}
Let $r\in \z$, $m: k\to k'$ is an embedding of perfect fields. 
Then the following statements are valid.
\begin{enumerate}
\item\label{ieff1} $\mshlam$ sends $\shelam(k)\{r\}$ into $\shelam(k')\{r\}$ and maps  $\shlamk_{\tshlam\ge r}$ into $\shlam(k')_{\tshlam\ge r}$.
\item\label{ieff2} $\mdmlam$ sends $\dmelam(k)\{r\}$ into $\dmelam(k')\{r\}$ and maps  $\dmlam(k)_{\tdmlam\ge r}$ into $\dmlam(k')_{\tdmlam\ge r}$.
\item\label{imk} 
 $\mklam$ sends $\shelam(k)\{r\}$ into $\dmelam(k)\{r\}$ and maps $\shlamk_{\tshlam\ge r}$ into $\dmlam(k)_{\tdmlam\ge r}$.

\item\label{iefftens} $\obj\shelam(k)\{r\}\otimes \obj\shelam(k)\subset \obj\shelam(k)\{r\}$ and $\shlam(k)_{\tshlam\ge r}\otimes \shlam(k)_{\tshlam\ge 0}\subset \shlam(k)_{\tshlam\ge r}$; $\obj\dmelam(k)\{r\} \otimes \obj \dmelam(k)\subset \obj \dmelam(k)\{r\} $ and
 $\dmlam(k)_{\tdmlam\ge r}\otimes \dmlam(k)_{\tdmlam\ge 0} \subset \dmlam(k)_{\tdmlam\ge r}$.

\item\label{ieff4} The correspondences $F\mapsto \obj\dmelam(F)\{r\}\cap \obj \dmlamc(F)$  and $F\mapsto \dmlam(F)_{\tdmlam\ge r}\cap \obj \dmlamc(F)$
for $F\in \obj \pfi$ are $\dmlamc$-continuous in the sense of Remark \ref{rcont}(2).\footnote{I.e., 
if $k=\inli k_i$ and $E\in \obj\dmelam(k)\{r\}\cap \obj \dmlamc(k)$ (resp.  $E\in \dmlam(k)_{\tdmlam\ge r}\cap \obj \dmlamc(k)$) then there exists 
 $k_0\in \{k_i\}$ along with some   $E^0\in  \obj\dmelam(k_0)\{r\}\cap \obj \dmlamc(k_0)$ (resp.  $E^0\in \dmlam(k_0)_{\tdmlam\ge r}\cap \obj \dmlamc(k_0)$) such that 
 $
 E^0_k\cong E$.} 

\item\label{itcons} The $t$-structures $\tshlam$ and $\tdmlam$ are non-degenerate.

\item\label{itt} The "forgetful" functors $F^{\sht}:\shlam(k)\to \sht(k)$ and $F^{\dm}: \dmlam(k)\to \dm(k)$ provided by Proposition \ref{plocoeff}(\ref{icg5})  are "strictly right $t$-exact", i.e., for $M\in \obj \shlam(k)$ (resp. $M\in \obj \dmlam(k)$)  we have $F^{\sht}(M)\in \sht(k)_{\tsh\ge 0}$ if and only if $M\in\shlam(k)_{\tshlam\ge 0}$ (resp. $F^{\dm}(M)\in \dm(k)_{\tdm\ge 0}$ if and only if $M\in \dmlam(k)_{\tdmlam\ge 0}$). 


\item\label{ieffinf} In the case $p>0$ assume in addition that $p\in S$. Then any infinitely effective object of $\dmlamc(k)$ (see Definition \ref{deff}(2)) is zero.

\item\label{ieffdual} Assume 
 once again that $S$ contains $p$ if $p>0$. Then the categories $\shlamc(k)$ and $\dmlamc(k)$ are rigid (i.e., all their objects are 
dualizable). Moreover, $\shlamc(k)$ is the smallest thick subcategory of $\shtlam(k)$ containing all $\sinftlam(P_+)\{i\}$ for $P$ being smooth projective over $k$ and $i\in \z$; $\dmlamc(k)$ is the smallest thick subcategory of $\dmlam(k)$ containing all  $\mglam(P)\{i\}$.

\item\label{ieffst} 
 All morphisms from $S^0_\lam$ into $\shlam(k)_{\tshlam\ge 1}$ are zero ones. 
\end{enumerate}

\end{pr}
\begin{proof}
\ref{ieff1},  \ref{ieff2}, \ref{imk}. By  definitions of the corresponding classes (see Definition \ref{dcomp}(2)  and Proposition \ref{paisle}), it suffices to note that $\mshlam$, $\mdmlam$, and $M_{k,\lam}$ are exact functors that respect small coproducts.

\ref{iefftens}. Since the tensor product bi-functors for $\shlam(k)$ and $\dmlam(k)$ respect co-products when one of the arguments is fixed and also "commute with $-\{i\}$", it suffices to note that $\sinftlam (-_+)(\sv) \otimes \sinftlam (-_+)(\sv) \subset \sinftlam (-_+)(\sv) $ and $\mglam(\sv) \otimes \mglam(\sv) \subset \mglam(\sv)$.

\ref{ieff4}. Obviously, we can  assume $r=0$.  
Next, for any perfect field $F$ Remark \ref{rcomp} implies that 
   $\dmelam(F)\cap \obj \dmlamc(F)$  is the smallest thick 
subcategory of $\dmlam(F)$ containing $\mglam(X)$ for all $X\in \sv(F)$. 
Moreover, Theorem 3.7 of \cite{postov} (as well as the more general Theorem 3.2.1(2) of \cite{bpws})
implies that $\dmlam(F)_{\tdmlam\ge 0}\cap \obj \dmlamc(F)$
is the $\dmlam(F)$-envelope (see  \S\ref{snotata}) of  $\mglam(X)\{j\}[l]$ for  $X$ running through $ \sv(F)$, $j\in\z$, and $l\ge 0$.
Hence the assertion follows from Remark \ref{rcont}(4). 

\ref{itcons}. The statement easily reduces to the case $\lam=\z$ in which it is well-known (see Lemma 5.5(2) of  \cite{degmod} and  \S5.2 of \cite{morintao}; cf. also Corollary 3.3.7(1) of \cite{bondegl}). 

\ref{itt}.  The assertion is rather easy; it follows immediately from Proposition 5.6.2(II.3) of \cite{bpgws}.

	\ref{ieffinf}. Immediate from Theorem 2.2 of \cite{binters} (see also Remark 2.3(2) of ibid. 
	and Proposition \ref{pteff}(4) below).
	
\ref{ieffdual}. Immediate from Theorem 2.4.8 of \cite{bondegl} (that relies on Appendix B of \cite{lyz}); cf. also Lemma 2.3.1 of \cite{bzp} and Proposition 5.5.3 of \cite{kellyth} where independent proofs of the $\dmlamc(k)$-part of the assertion were given.

\ref{ieffst}. This is a well-known statement that 
  can be easily obtained from Example 5.2.2 of \cite{morintao}.
\end{proof}


We will also need  the effective versions of our homotopy $t$-structures along with some of their properties.

\begin{defi}\label{defft}
 1. Denote by $\tshelam$ (resp. $\tdmelam$) the $t$-structure on $\shelam(k)$ (resp. on $\dmelam(k)$) generated by $\sinftlam(X_+)$ (resp. by $\mglam(X)$) for $X\in \sv$. 

2. Denote by $i^{\sht}_\lam=i^{\sht}_{\lam,k}$ (resp. by $i^{\dm}_\lam=i^{\dm}_{\lam,k}$) the embedding $\shelam(k)\to \shlam(k)$. Their right adjoints (see Remark \ref{rtgen}(3)) will be denoted by $w^{\sht}_\lam$ and  $w^{\dm}_\lam$, respectively.

Omitting $k$, let us 
denote the compositions $i^{\sht}_\lam\circ w^{\sht}_\lam$  and $i^{\dm}_\lam\circ w^{\dm}_\lam$ by $\nu_{\shtlam}^{\ge 0}$ and $\nu_{\dmlam}^{\ge 0}$, respectively. Moreover, for any $r\in \z$ we will consider the functors  
$\nu_{\shtlam}^{\ge r}= (\nu_{\shtlam}^{\ge 0}(-\{-r\}))\{r\}$ and $\nu_{\dmlam}^{\ge r}= \nu_{\dmlam}^{\ge 0}(-\{-r\}))\{r\}$. 

3. For a homological functor $H$ from $\shlam(k)$ (resp. from  $\dmlam(k)$) with values in some abelian category  the symbol $\filt^rH$ will (similarly to \cite{levconv}) denote   the functor $E\mapsto \imm(H(\nu_{\shtlam}^{\ge r}(E))\to H(E))$ (resp.   $E\mapsto \imm(H(\nu_{\dmlam}^{\ge r}(E))\to H(E))$; here the connecting morphisms are induced by the corresponding counit transformations). 
\end{defi}


The following statements appear to be (quite easy and) rather well-known.

\begin{pr}\label{pteff}
1. The functors $i^{\sht}_\lam$ and $i^{\dm}_\lam$ are right $t$-exact with respect to the corresponding $t$-structures, whereas their right adjoints are $t$-exact.

Moreover, the compositions $\nu_{\shtlam}^{\ge 0}$  and $\nu_{\dmlam}^{\ge 0}$ respect coproducts. 

2. Denote  $\nu_{\shtlam}^{\ge 1}(S^0_\lam)$ by $\emlamo$. Then $\emlamo$ belongs to $\shelam(k)_{\tshelam\ge 0}$, and there exists a distinguished triangle $\emlamo\to S^0_\lam\to \emlam\to \emlamo[1]$, where  $\emlam$ is the image of $\emz$ (see Proposition \ref{pcont}(\ref{ishdmmg}))   in $\shtlam$.


3. The $\lam$-linear analogue $U_{k,\lam}$ of $U_k$ sends  $\dmelam(k)\{r\}$ into $\shelam(k)\{r\}$ and maps $\dmlam(k)_{\tdmlam\ge r}$ into $\shlamk_{\tshlam\ge r}$ (for any $r\in \z$).

4. Assume in addition that $S$ contains $p$ whenever $p>0$. Then for any  $M\in\obj \dmlamc(k)$ there exists $r\in \z$ such that $\nu_{\dmlam}^{\ge r}(M)=0$.
\end{pr}
\begin{proof} 1. The first part of the statement is obvious (cf. the proof of Proposition \ref{peff}(\ref{ieff2})). 

Its second part can be proved similarly to  Corollary 3.3.7(2) of \cite{bondegl} (and follows from it in "most" cases; in the remaining cases the arguments of loc. cit. may be combined with Theorem 5.2.6 of \cite{morintao}). 

Lastly, Remark \ref{rtgen}(1) (cf. also part 3 of that remark) immediately yields that the functors $\nu_{\shtlam}^{\ge 0}$  and $\nu_{\dmlam}^{\ge 0}$ respect coproducts in the corresponding categories.

2. $\nu_{\shtlam}^{\ge 1} (S^0_{\lam})$ belongs to $\shelam(k)_{\tshelam\ge 0}$ immediately from the previous assertion.

Next, in the case $\lam=\z$ the (existence and the)
 properties of the distinguished triangle in question 
 immediately follow from Theorem 10.5.1 of \cite{levht}; clearly, the general case follows from this one.

3. Similarly to the proof of Proposition \ref{peff}, it suffices to "control" $U_{k,\lam}(\mglam(X))$ for $X\in \sv$.
We clearly have $U_{k,\lam}(\mglam(X))\cong \emlam\otimes \sinftlam(X_+)$. The previous assertion obviously implies that $\emlam\in \shelam(k)_{\tshelam\ge 0}$; thus Proposition \ref{peff}(\ref{iefftens}) yields the result.

4. Immediate from Lemma 
 2.7 of \cite{binters} (whose proof is based on an argument from \cite{ayconj}).
\end{proof}

\subsection{
On the "homotopy 
conservativity" of  motivization 
}\label{sdiff}


Now we are able to prove that certain restrictions of $M_{k,\lam}^c$ "strictly respect homotopy connectivity"; this statement significantly strengthens 
 Theorem \ref{peasy}(I.2,II).  The reader may consult  
 sections \ref{sshcont} and \ref{sthom}  for the corresponding definitions. Note also that one can certainly take $\lam=\z$ in the following theorem.

\begin{theo}\label{tbach}

Let $E \in \obj \shlamc(k)\setminus  \shlam(-)_{\tshlam\ge r}$ for some $r\in \z$.
Then $M_{k,\lam}(E)\notin \dmlam(k)_{\tdmlam\ge r}$ (one may say that $E$ is not {\it $r-1$-homotopy connective}) whenever either (i)  $k$ is non-orderable or (ii) $E$ is $2$-torsion.  \end{theo}
\begin{proof}

First we assume that $k$ is non-orderable.
 Then once again (by Corollary \ref{cohdim}; cf. the proof of 
 Theorem \ref{peasy}(I.2)) we can present $k$ as $\inli k_i$, where the  cohomological ($2$)-dimensions of $k_i$ are finite.
Now      recall that the functor $F\mapsto  \dmlam(F)_{\tdmlam\ge r} \cap \obj \dmlamc(F)$ (from $\pfi$ into sets)  is $\dmlamc$-continuous; 
 see Proposition \ref{peff}(\ref{ieff4}). Hence Remark \ref{rcont}(3)  (combined with Proposition \ref{peff}(\ref{ieff1})) enables us to assume that the cohomological dimension of $k$ is finite.

Now, under this additional assumption the  $\lam=\z$-case of our assertion is given by 
Theorem 16(b) of \cite{bachcons}. 
In the general case we note that $E$ may be considered as an object of $\sht(k)$ via the embedding $G$ 
 mentioned in Proposition \ref{plocoeff}(\ref{icg5}); $G(E)$ is clearly homotopy connective and slice-connective in $\sht(k)$ (see Remark \ref{reff}(1)). Hence this case of our assertion follows from  version (i) of loc. cit. combined with Proposition \ref{peff}(\ref{itt}). 

Lastly, in the case (ii) we argue similarly to the proof of 
 Theorem \ref{peasy}(II). Consequently, we can (and will) assume that $k$ is a finitely generated field of characteristic $0$. Once again, $E$ yields a ($2$-torsion) homotopy connective and slice-connective object of  $\sht(k)$. So (after we invoke Proposition \ref{peff}(\ref{itt})) it remains to apply Lemma 19 of ibid.

\end{proof}

\begin{coro}\label{ceasylam}
If $k$ is non-orderable then the motivization functor $M_{k,\lam}^c$ is conservative. Moreover, the restriction of  $M_{k,\lam}^c$ to the subcategory of $2$-torsion objects is conservative for any (perfect) $k$.  
\end{coro}
\begin{proof}
Obviously, this statement is equivalent to the $\lam$-linear version of 
 Theorem \ref{peasy}(I.2,II). Hence for $E \in \obj \shlamc(k)$ such that $M_{k,\lam}(E)=0$
we should check that $E=0$ whenever either   $k$ is non-orderable or  $E$ is $2$-torsion. Now, if $E\neq 0$ then Proposition \ref{peff}(\ref{itcons}) gives the existence of an integer $r$ such that 
$E \notin \shlam(-)_{\tshlam\ge r}$. Hence 
the assertion follows from Theorem \ref{tbach}.
\end{proof}

Combining 
 this corollary with 
 a theorem from \cite{bachinv}, we easily obtain the following result (slightly generalizing another Bachmann's statement).

\begin{pr}\label{pquadr}
Assume $p\neq 2$, $k$ is non-orderable,  and 
  $S$ contains $p$ if $p>0$. Let $\phi$ be a non-zero $k$-quadratic form and $a\in k\setminus \ns$. Then for the affine variety $X$ given by the equation $\phi=a$ the object $C=\co(\sinftlam(X_+)\to S^0_\lam)$ (corresponding to the structure morphism for $X$) is $\otimes$-invertible in  $\shlam(k)$.
\end{pr}
\begin{proof}

Firstly note that $\phi$ may be assumed to be non-degenerate. Indeed, if the kernel of (the symmetric  bilinear form corresponding to) $\phi$ is of dimension $j\ge 0$ and $\phi'$ is the corresponding non-degenerate form then 
$X$ is isomorphic to the product of the zero set $X'$ of $\phi'-a$ by the affine space $\af^j$. Thus we have $\sinftlam(X'_+)\cong  \sinftlam(X_+)$ (see 
Proposition \ref{pcont}(\ref{ishdmsinf})).

Next,  $C\in \obj \shlamc(k)$; hence it is dualizable (see Proposition \ref{peff}(\ref{ieffdual})). Thus we should check whether the evaluation morphism $C\otimes C^{\vee}\to S^0_\lam$ is invertible (where $C^{\vee}$ is the dual to $C$). Since 
  $\mklamc$ is symmetric monoidal and also conservative (in this case), it suffices to verify that a similar fact is valid in $\dmlamc(k)$. The latter is immediate from Theorem 33 of \cite{bachinv} (where $\phi$ was assumed to be non-degenerate).
\end{proof}

\begin{rema}\label{rquadr} 
1. Certainly, it does not make much sense to consider $S\not \subset \{p\}$ in this statement.

2. Actually, in the introduction to \cite{bachcons} it is said $C$ is $\otimes$-invertible in  $\sht(k)$  also in the case of a formally real $k$; this statement appears to follow from the results of ibid. easily.  Clearly, our continuity arguments reduce this fact to the case where $k$ is a finitely generated 
  field.
\end{rema}

The following generalization of Theorem 2.2.1 of \cite{asok} follows easily as well. 

\begin{pr}\label{pasok}
Assume that $k$ is non-orderable; let $X/k$ be a smooth 
 variety.
Then the following statements are equivalent.

\begin{enumerate}


\item\label{iconn1}
The morphism $\sinftlam(X_+)\to S^0_\lam$  (induced by the structure morphism $X\to \spe k$) 
 gives $H_0^{\tshlam}(\sinftlam(X_+))\cong H_0^{\tshlam}(S^0_{\lam})$.

\item\label{iconn2}
We have a similar isomorphism $H_0^{\tshelam}(\sinftlam(X_+))\to H_0^{\tshelam}(S^0_{\lam})$ (here we consider $\sinftlam(X_+)$ and $S^0_\lam$ as objects of $\shelam(k)$).

\item\label{iconn3}
  $H_0^{\tdmlam}(\mglam(X))\cong H_0^{\tdmlam}(\lam)$.

\item\label{iconn4}
  $H_0^{\tdmelam}(\mglam(X))\cong H_0^{\tdmelam}(\lam)$.
\end{enumerate}

 
\end{pr}
\begin{proof}
Note that $\sinftlam(X_+)$ and   $S^0_\lam$ belong to $\shelam(k)_{\tshelam\ge 0}$, whereas $\mglam(X)$ and $\lam$ belong to  $\dmelam(k)_{\tdmelam\ge 0}$. Thus  Proposition \ref{pteff} easily implies that condition  \ref{iconn1} is equivalent to condition  \ref{iconn2}, and   \ref{iconn3} is equivalent to  \ref{iconn4}.

 Now assume that $H_0^{\tdmlam}(\mglam(X))\cong H_0^{\tdmlam}(\lam)$. Applying  Theorem \ref{tbach} (version (i)) in the case $r=1$, $E=\co (\sinftlam(X_+)\to S^0_\lam)$, we obtain that $E\in \shlam(k)_{\tshlam\ge 1}$. Applying Proposition \ref{peff}(\ref{ieffst}) (and considering the exact sequence $\dots \to \shlam(k)(S^0_\lam,\sinftlam(X_+))\to  \shlam(k)(S^0_\lam,S^0_\lam)\to \shlam(k)(S^0_\lam, E)=\ns$) we obtain a splitting $\sinftlam(X_+)\cong S^0_\lam \bigoplus E[1]$. Thus the application of  Theorem \ref{tbach} to our $E$ 
 also yields that our  condition  \ref{iconn3} implies  condition  \ref{iconn1}.

Lastly, applying this splitting argument we  easily obtain that condition \ref{iconn1} implies condition \ref{iconn3}; one should only apply Proposition \ref{peff}(\ref{imk}) instead of Theorem \ref{tbach}.
\end{proof}

\begin{rema}\label{rasok}

1. Surely, for 
$E$ as in this proof one can also apply Theorem \ref{tbach} to the spectrum $E/2=\co(E\stackrel{\times 2}{\to} E)$
(without having to assume that $k$ is non-orderable).

2. Now assume in addition that $X$ is (also) proper.

Then one can easily see (cf. Lemma 2.1.3 of \cite{asok}; here $k$ may be any perfect field) that condition \ref{iconn4} (and \ref{iconn3}) of our Proposition is fulfilled if and only if 
the kernel of the degree homomorphism
 $\operatorname{Chow}_{0}(X_L)\to \z$ is $S$-torsion
 and $X_L$ contains a zero-cycle whose degree 
 is a product of elements of $S$ for any field extension $L/k$.
Obviously, it suffices to verify the latter condition for $L=k$ only.

3. Under the assumption that $p$ belongs $S$ whenever it is positive one may formulate a much more general result of this sort.
Indeed, Corollary 3.3.2 of \cite{bscwh} gives for $M\in \obj \dmelamc(k)$ 
several  conditions equivalent to  $M\in \dmelam(k)_{\tdmelam\ge 0}$ (note that in ibid. the cohomological convention for $t$-structures is used). Most of these conditions   
are formulated in terms of the so-called Chow-weight homology of $M$. Thus 
 assuming that $E\in \obj \shelamc(k)$ is $2$-torsion whenever
$k$ is formally real, one obtains an answer to the question whether $\mklam(E)$ belongs to $\shelam(k)_{\tshelam\ge 0}$ in terms of certain complexes of Chow groups corresponding to $E$.

4. Clearly,  in the case $\lam=\z$ and $X$ being  proper we could have deduced (most of) our proposition directly from  Theorem 2.2.1 of \cite{asok} (using the $\dmc$-continuity of $\dm(-)_{\tdm\ge 0}\cap \obj \dmc(-)$).
\end{rema}

\section{
On 
infinite effectivity and 
other supplements }\label{skernel}

This section is dedicated to those results on the motivization kernel that are not (closely) related to the ones in the literature.

In \S\ref{sinfeff} we prove that the compact motivization functor $\mklam^c$ "strictly respects" the slice filtrations (on $\shlamc(k)$ and $\dmlamc(k)$, respectively); it also "detects" the filtration $\filt^*$ (see Definition \ref{defft}) on the lower $\thom$-homology of an object of $\shlamc(k)$.

 In \S\ref{salt}
 we describe an alternative method of the proof of 
Theorem \ref{tbach}(i) 
 (under the additional assumption that $p$ is invertible in $\lam$ whenever it is positive). 

In \S\ref{scobormod}
 we  explain that in all our results the 
 categories $\dmlam(-)$ may be replaced by the ($\lam$-linearized) categories $\shmglam(-)$ of  (strict) modules over the Voevodsky's motivic cobordism spectrum $\mgl$.

\subsection{The effectivity description of the "compact motivization kernel"}\label{sinfeff}

Now we prove "the most original" result of this paper. Recall that $\lam$ is  the coefficient ring for our motivic categories (it is an arbitrary localization of $\z$),  $M_{k,\lam}$ denotes the $\lam$-linear version of the "motivization" functor from the motivic stable homotopy category $\sht$ to $\dm$, $\shelam(k)\{r\}$ and  $\dmelam(k)\{r\}$  denote  the $r$-th levels of the slice filtrations on the corresponding categories, whereas  $\tshlam$ and $\tdmlam$ are the corresponding homotopy $t$-structures.

\begin{theo}\label{tineffkerm}
Let $r,m\in \z$, and $E$ is an object of $\shlam$.

I. Denote by $\shlam(k)_{\tshlam\ge m}^{\{r\}}$ the smallest class of objects of $\shlam(k)$ that  is stable with respect to extensions and arbitrary (small) coproducts, and contains $\sinftlam(X_+)\{j\}[i]$ whenever $i>m$ or if $i=m$ and $j\ge r$.

1. For any $r',m'\in \z$ we have $\shlam(k)_{\tshlam\ge m}^{\{r\}} \otimes \shlam(k)_{\tshlam\ge m'}^{\{r'\}}\subset \shlam(k)_{\tshlam\ge m+m'}^{\{r+r'\}}$. 

2. $\shlam(k)_{\tshlam\ge m}^{\{r\}}\subset \shlam(k)_{\tshlam\ge m}$, and  $\shlam(k)_{\tshlam\ge m}^{\{r\}}$ contains the classes 
 $\shlam(k)_{\tshlam\ge m+1}$ and $ \shelam(k)_{\tshelam\ge m}\{r\}=\nu_{\shtlam}^{\ge r}(\shlam(k)_{\tshlam\ge m})$ (see Definition \ref{defft}).

3.  The following conditions are equivalent. 

(i)  $E$ belongs to   $\shlam(k)_{\tshlam\ge m}^{\{r\}}$.  

(ii)  
$E$ belongs to $ \shlam(k)_{\tshlam\ge m}$ and $\filt^r H^{\tshlam}_m(E) $  (see Definition \ref{defft}(3)) equals the whole 
  object $H^{\tshlam}_m(E) $. 

(iii) $E$ is an extension of an element of $ \shlam(k)_{\tshlam\ge m+1}$ by an element of $ \shelam(k)_{\tshelam\ge m}\{r\}$.

4. Define $ \dmlam(k)_{\tdmlam\ge m}^{\{r\}}$ as the smallest class  of objects of $\dmlam(k)$ that  is stable with respect to extensions and arbitrary (small) coproducts, and contains  $\mglam(X)\{j\}[i]$ whenever $i>m$ or if $i=m$ and $j\ge r$.

Then an object $M$ of $\dmlam$  belongs to  $ \dmlam(k)_{\tdmlam\ge m}^{\{r\}}$ if and only if $M\in \dmlam(k)_{\tdmlam\ge m}$ and  $\filt^r H^{\tdmlam}_m(M)=  H^{\tdmlam}_m(M_{k,\lam}(M))$. 

Moreover, these conditions are fulfilled if and only if $M$ is  an extension of an element of  $\dmlam(k)_{\tdmlam\ge m+1}$ by an element of $ \dmelam(k)_{\tdmelam\ge m}\{r\}$.

5. $M_{k,\lam}(\shlam(k)_{\tshlam\ge m}^{\{r\}})\subset  \dmlam(k)_{\tdmlam\ge m}^{\{r\}}$ and $U_{k,\lam}( \dmlam(k)_{\tdmlam\ge m}^{\{r\}})\subset \shlam(k)_{\tshlam\ge m}^{\{r\}}$.

 
II. Assume that $E$ is slice-connective (i.e., belongs to $\obj\shelam\{i\}$ for some $i\in \z$). 

Then the following statements are valid.

1. $E\in \obj \shelam(k)\{r\}$ if and only if $M_{k,\lam}(E)\in \obj \dmelam(k)\{r\}$. In particular, $E\in \cap_{j\in \z}  \obj \shelam(k)\{j\}$ (i.e., it is infinitely effective) if and only if $M_{k,\lam}(E)$ also is.

2. Assume that $E\in  \shlam(k)_{\tshlam\ge m}$. Then  $E$ also belongs to $\shlam(k)_{\tshlam\ge m}^{\{r\}}$ (see assertion I.1) 
 if and only if  $\filt^r H^{\tdmlam}_m(M_{k,\lam}(E))=  H^{\tdmlam}_m(M_{k,\lam}(E))$ (cf. assertion I.4).

III. Assume in addition that $E\in  \obj \shlamc(k)$; if $p>0$ then suppose also that $p\in S$. 

1.  $E$ 
is infinitely effective if and only if $M_{k,\lam}(E)=0$.

2. If 
$E\in  \shlam(k)_{\tshlam\ge m}$ then $M_{k,\lam}(E)\in \dmlam(k)_{\tdmlam \ge m+1}$ if and only if $E$ also belongs to $\shlam(k)_{\tshlam\ge m}^{\{s\}}$ for all $s\in\z$. 
\end{theo}
\begin{proof}

I. Obviously, one can assume $m=r=0$.

1.  
 It suffices to recall that the tensor product on $\shlam(k)\subset \sht(k)$ respects coproducts and distinguished triangles, and $\sinftlam(X_+)\otimes \sinftlam(Y_+)\cong \sinftlam(X\times Y_+)$ for any smooth $k$-varieties $X$ and $Y$. 

2. 
$\shlam(k)_{\tshlam\ge 1}\subset \shlam(k)_{\tshlam\ge 0}^{\{0\}}\subset \shlam(k)_{\tshlam\ge 0}$ immediately from the description of $\shlam(k)_{\tshlam\ge 0}$ provided by Proposition \ref{paisle} (see also Definition \ref{deff}(1)). Similarly, the class $\shlam(k)_{\tshlam\ge 0}^{\{0\}}$ contains $\shelam(k)_{\tshelam\ge 0}$; see Definition \ref{defft}(1).

Lastly, recall that $\nu_{\shtlam}^{\ge 0}=i^{\sht}_\lam\circ w^{\sht}_\lam$. The functor $w^{\sht}_\lam$ is $t$-exact by Proposition \ref{pteff}(1); hence $w^{\sht}_\lam(\shlam(k)_{\tshlam\ge 0})\subset \shelam(k)_{\tshelam\ge 0}$. It remains to recall that the restriction of the functor $w^{\sht}_\lam$ to $\shelam(k)\subset \shlam(k)$ is  the identity; hence 
 the class $ \shelam(k)_{\tshelam\ge 0}\{0\}$ equals $\nu_{\shtlam}^{\ge 0}(\shlam(k)_{\tshlam\ge 0})$ indeed.


3. 
Applying assertion I.2 we immediately obtain that condition (iii) implies condition (i). 

Now assume that $E$ fulfils  condition (ii).
Assertion I.2 says that $\nu_{\shtlam}^{\ge 0}(E)\in \shelam(k)_{\tshelam\ge 0}$. 
 We take the obvious distinguished triangle
\begin{equation}\label{esl}
 \nu_{\shtlam}^{\ge 0}(E)\to E\to C\to \nu_{\shtlam}^{\ge 0}(E)[1];\end{equation}
and obtain that that $E$  satisfies condition (iii) whenever $C\in \shlam(k)_{\tshlam\ge 1}$.
Consider the long exact sequence $\dots\to H^{\tshlam}_0(\nu_{\shtlam}^{\ge 0}(E))\stackrel{f}{\to} H^{\tshlam}_0(E)\to H^{\tshlam}_0(C)\to H^{\tshlam}_{-1}(\nu_{\shtlam}^{\ge 0}(E))\to \dots$ coming from (\ref{esl}). Condition (ii) implies that $f$ is epimorphic. Since  $H^{\tshlam}_{-1}(\nu_{\shtlam}^{\ge 0}(E))=0$, we obtain that   $H^{\tshlam}_0(C)$ is zero as well. Since (\ref{esl}) also implies that  $C\in \shlam(k)_{\tshlam\ge 0}$, we  obtain that  $C\in \shlam(k)_{\tshlam\ge 1}$ indeed.

Lastly, to prove that all elements of  $ \shlam(k)_{\tshlam\ge 0}^{\{0\}}$ satisfy condition (ii) (i.e., to prove that (i) implies (ii)) it clearly suffices to  prove that this condition is fulfilled for all elements of $ \shlam(k)_{\tshlam\ge 1}$ and  $ \shelam(k)_{\tshelam\ge 0}$, and the class of spectra satisfying  condition (ii)  is closed with respect to coproducts and extensions. Now, the first of these statements is obvious, and the second one immediately follows from the fact that the functors $H^{\tshlam}_0$ and $\nu_{\shtlam}^{\ge 0}(E)$ respect coproducts (see Remark \ref{rtgen}(1) and Proposition \ref{pteff}(1)). 

	Thus it remains  to verify for any $\shlam$-distinguished triangle \begin{equation}\label{etr}
 A\to B\to C\to A[1]\end{equation} such that $A$ and $C$ satisfy condition (ii) that  the spectrum $B$ satisfies this condition as well. Since the class $\cu_{t\ge 0}$ is extension-closed for any $t$-structure on a triangulated category $\cu$, it is sufficient to check that $H^{\tshlam}_0( \nu_{\shtlam}^{\ge 0}(B))$ surjects onto $H^{\tshlam}_0(B)$. Now, we apply $H^{\tshlam}_0$ to the distinguished triangle (\ref{etr}) and its image under the exact endofunctor $\nu_{\shtlam}^{\ge 0}$ to obtain the following commutative diagram with exact rows:
$$\begin{CD}
 H^{\tshlam}_{0}(\nu_{\shtlam}^ {\ge 0}(A))@>{}>>H^{\tshlam}_{0}(\nu_{\shtlam}^ {\ge 0}(B)) @>{}>>H^{\tshlam}_{0}(\nu_{\shtlam}^ {\ge 0}(C)) @>{}>> H^{\tshlam}_{-1}(\nu_{\shtlam}^ {\ge 0}(A)) \\
@VV{c}V@VV{d}V@VV{e}V@VV{}V \\
H^{\tshlam}_{0}(A) @>{}>>H^{\tshlam}_{0}(B) @>{}>>H^{\tshlam}_{0}(C) @>{}>>H^{\tshlam}_{-1}(A)
\end{CD}$$
 The spectrum $\nu_{\shtlam}^ {\ge 0}(A)$ 
  belongs to $ \shlam(k)_{\tshlam\ge 0}^{\{0\}}\subset  \shlam(k)_{\tshlam\ge 0}$ according to assertion I.2;  hence  $H^{\tshlam}_{-1}(\nu_{\shtlam}^ {\ge 0}(A))=0$. 
  Thus the surjectivity of $c$ and $e$ implies that of $d$, and we conclude the proof.

4. The proof is the obvious $\dmlam$-version of that of assertion I.3.

5. The statement immediately follows from the two previous assertions since $M_{k,\lam}(\shlam(k)_{\tshlam\ge m})\subset  \dmlam(k)_{\tdmlam\ge m}$,  $M_{k,\lam}(\shelam(k)\{r\}) \subset \dmelam\{r\}$,   $U_{k,\lam}(\dmelam\{r\})\in \shelam(k)\{r\}$,  and $U_{k,\lam}(\dmlam(k)_{\tdmlam\ge m+1})\subset\shlamk_{\tshlam\ge r}$; see Proposition \ref{peff}(\ref{imk}) and Proposition \ref{pteff}(3).

II.1. Proposition \ref{peff}(\ref{imk}) immediately gives the "only if" implication.

Now we prove the "if" part of the assertion; so we assume that $M_{k,\lam}(E)\in \obj \dmelam(k)\{r\}$.

By the definition of slice-effectivity, $E$ belongs to $\obj\shelam(k)\{r'\}$ for some $r'\in \z$; take the maximal 
 $r'\le r$ such that this inclusion is fulfilled (in particular, we take $r'=r$ if  $E$ belongs to $\obj\shelam(k)\{r''\}$ for some $r''\ge r$, since $\shelam(k)\{r\}\subset \shelam(k)\{r''\}$ in this case). 

Consider the distinguished triangle \begin{equation}\label{eukmk}
\emlamo\otimes E\to E\to \emlam\otimes E= U_{k,\lam}(M_{k,\lam}(E))\to \emlamo\otimes E[1];
\end{equation} 
  see Proposition \ref{pteff}(2). Recall that $\emlamo=\nu_{\shtlam}^{\ge 1}(S^0_\lam) \in \obj\shelam(k)\{1\}$; hence $\emlamo\otimes E\in \obj\shelam(k)\{r'+1\}$ (see Proposition \ref{peff}(\ref{iefftens})). 
	Since $U_{k,\lam}(M_{k,\lam}(E))\in \obj\shelam(k)\{r\}$ (see Proposition \ref{peff}(3)), we obtain that 
  $E\in \obj\shelam(k)\{r'+1\}$ whenever $r'<r$. Thus $r'=r$.  

2.  The "only if" statement is given by assertion I.5.

The proof of the converse implication is similar to that of the previous assertion.
 There clearly exists $r'\in \z$ such that $\filt^{r'} H^{\tshlam}_m(E) = H^{\tshlam}_m(E) $ and $r'\le  r$, and we choose the maximal  $r'\le r$ 
 such that this equality is fulfilled. According to assertion I.3, $E$ belongs to $\shlam(k)_{\tshlam\ge m}^{\{r'\}}$. 

Now we look  at the triangle (\ref{eukmk}). Since $\emlamo=\nu_{\shtlam}^{\ge 1}(S^0_\lam)$ (see Proposition \ref{pteff}(2)) and $S^0_\lam\in \shlam(k)_{\tshlam\ge 0}$, 
 we obtain $\emlamo\in \shlam(k)_{\tshlam\ge 0}^{\{1\}}$ by assertion I.2; hence $\emlamo\otimes E\in \shlam(k)_{\tshlam\ge m}^{\{r'+1\}}$ according to assertion I.1. Next, the spectrum $\emlam\otimes E= U_{k,\lam}(M_{k,\lam}(E))$ belongs to  $\shlam(k)_{\tshlam\ge m}^{\{r\}}$ according to assertion I.5. Thus if $r'<r$ then $E$ belongs to $\shlam(k)_{\tshlam\ge m}^{\{r'+1\}}$ since this class is extension-closed by definition. Therefore $r'=r$ (see assertion I.3 once again).

III.1.  If $M_{k,\lam}(E)=0$ then  $E$ is infinitely effective according to  assertion II.1. 

Conversely, since $M_{k,\lam} (\obj\shelam(k)\{s\})\subset \obj \dmelam(k)\{s\}$ for any $s\in \z$, we obtain that $M_{k,\lam}$ respects infinite effectivity. Lastly,  Proposition \ref{peff}(\ref{ieffinf}) says that there are only zero infinitely effective compact objects in $\dmlam$.

2. According to assertion II.2, for any $s\in \z$ we have  $\filt^s H^{\tshlam}_m(E) = H^{\tshlam}_m(E) $  if and only if $\filt^s H^{\tdmlam}_m(M_{k,\lam}(E))=  H^{\tdmlam}_m(M_{k,\lam}(E))$. Now, by Proposition \ref{pteff}(4), 
the latter is equivalent to $H^{\tdmlam}_m(M_{k,\lam}(E))=0$ (if we take $s$ to be large enough); combined with Proposition \ref{peff}(\ref{imk}) this yields the result.
\end{proof}

\begin{rema}\label{rwc}
1. So we obtain that $\mklam^c$ induces an exact  conservative functor from the localization of $\shlamc(k)$ by its subcategory of infinitely effective objects into $\dmlamc(k)$ (under the assumption that   $p\in S$). 

Now, recall that the same restriction on $S$ ensures the existence of an exact conservative {\it weight complex} functor $ \dmlamc(k)\to K^b(\chowlam(k))$
(that was essentially constructed in \cite{bws}  for $p=0$ and in \cite{bzp} in the case $p>0$; see \cite[Propositions 3.1.1, 2.3.2]{bonivan} for the $\lam$-linear formulation).
Thus the composition functor  is conservative as well; if $k$ is non-orderable this is actually a functor $\shlamc(k)\to  K^b(\chowlam(k))$ (by Corollary \ref{ceasylam}).

2. Note however  that $\eta\neq 0$ unless $2\in S$ (by Theorem 6.3.3 
of \cite{morintao}); hence there cannot exist a {\it Chow weight structure} on $\shlamc(k)$ in this case (i.e., the motivic spectra of smooth projective varieties cannot belong to the heart of any weight structure on $\shlamc(k)$;  see the easy 
 Remark 5.2.7(6) of \cite{bgn}) and this composed version of the weight complex functor does not come from a weight structure.

On the other hand,  in 
\cite{bokum} 
an interesting  weight structure $w_{\chow}^{eff}$ on $\dmelam(k)$ that is {\it generated} by motives of all smooth varieties was considered. 
  $w_{\chow}^{eff}$ "naturally" extends to $\dmlam(k)$; the heart of the resulting weight structure $w_{\chow}$ naturally contains the category of Chow motives. Now, it may make sense to consider  similar definitions for $\shelam(k)\subset \shlam(k)$; yet the hearts of the resulting weight structures 
	 probably will not be closely related to Chow motives.   


3. It appears that some of the parts of our theorem can be extended to 
certain filtrations on $\shtlam(k)$ and $\dmlam(k)$ distinct from that given by $\shlam(k)_{\tshlam\ge m}^{\{r\}}$ and $\dmlam(k)_{\tdmlam\ge m}^{\{r\}}$; cf. \S3.3 and Remark 5.3.3(8) of \cite{bscwh}.
\end{rema}


\subsection{An alternative 
argument for Theorem \ref{tbach}(i)} 
\label{salt}

Now we describe an alternative proof of 
version (i) of Theorem \ref{tbach} 
that relies on 
  the paper \cite{levconv} instead of \cite{bachcons} (that is more complicated and   based on related results of M. Levine as well). This reasoning requires us to assume that $p\in S$ whenever $p>0$. Note also that this version  of our theorem
 obviously implies the 
$\zop$-linear version of 
 Theorem \ref{peasy}(I.2).

So, we suppose that $k$ is non-orderable. Then the continuity argument used in the proof of Theorem \ref{tbach}  allows us to assume (once again) that the cohomological dimension of $k$ is finite.

According to 
Theorem \ref{tineffkerm}(III.2) (combined with Remark \ref{reff} and Proposition \ref{peff}(\ref{imk})),
 it suffices to prove (under our assumption on $S$ and for some $m\in \z$) that the filtration $\filt^*$ on $H^{\tshlam}_m(E)$ is 
non-trivial (i.e., that $H^{\tshlam}_m(E)$ does not lie in its own $\filt^s$ for all $s\in \z$) for any $E$ belonging to $ \obj \shlamc(k)\cap \shlam(k)_{\tshlam\ge m}\setminus  \shlam(k)_{\tshlam\ge m+1}$. 

Next, the  well-known 
Proposition 5.1.1(5) of \cite{bgn} (that is an easy consequence of \cite[Lemma 4.2.7]{morintao})
implies the following: it suffices to verify that the filtration in question is "separated at function field stalks", i.e., that for any finitely generated field $L/k$ and $j\in \z$ the filtration induced by    $\filt^*H^{\tshlam}_m(E)$ on the result of the "evaluation of $E\{j\}$ at $L$" (see \S3.2.1 of \cite{degorient}) is separated.

Once again, we consider  $E$ as an object of $\shtk$ using the embedding $G$ described in Proposition \ref{plocoeff}(\ref{icg5}). Then it is {\it cohomologically finite} in the sense of  Definition 6.1 of \cite{levconv}. Indeed, 
$E$ satisfies condition (ii) of loc. cit. since it 
 is homotopy connective (see Remark \ref{reff}). It satisfies condition (i) of the definition according to  Proposition 6.9(3) of ibid. combined with Proposition \ref{peff}(\ref{ieffdual}) above. 
Hence the separatedness in question is given by Theorem 7.3 of \cite{levconv}; this finishes the proof.


\begin{rema}\label{rcompy}
1. The "yoga" of this argument (as well as of  Theorem \ref{tbach} itself) is that certain types of assertions concerning {\bf compact} objects of $\shlam$ can be reduced to the case where the (virtual) cohomological dimension of $k$ is finite. Now, under this additional assumption one can apply the appropriate properties 
(as studied Bachmann and Levine) of certain subcategories of $\shtk$ 
 that are bigger than $\shtc(k)$ (or $\shlamc(k)$ for the corresponding $\lam$).    So, in our main statements we  restrict ourselves to compact motivic spectra; this enables us to 
 establish them over a wide class of base fields. This method appears to be quite useful since (most of)  motivic spectra "coming from geometry" are compact. Also, one "usually" does not apply $\mk$ to non-compact objects of $\sht$ (that are mostly used for representing various cohomology theories).


2. One may also apply some of the arguments of  \cite{levconv} for proving 
 version (ii) of  Theorem \ref{tbach}.

\end{rema}

\subsection{On  cobordism-module versions of the main results}\label{scobormod}

Now recall that the category  $\dmk$ is closely related to the homotopy category of highly structured modules over the ring object $\emz$ in 
the model category of motivic symmetric spectra underlying $\sht$ (see  Proposition 38 of \cite{roe}).

The goal of this section is to explain  that our main results are also valid if we replace $M_{k}$ (and $M_{k,\lam}$) by the corresponding functor $M^{\mgl}_{k}:\shtk\to \shmgl(k)$, where the latter is the (stable) homotopy category of the category $\mglmod$ of (strict left) modules over the Voevodsky's 
 spectrum $\mgl$ (see \S1.3 of \cite{bondegl}). Similarly to \cite{roe}, one can verify the existence of  $M^{\mgl}_{k}$ given by the  "free $\mgl$-module functor"; the corresponding forgetful functor yields the right adjoint   $U^{\mgl}_{k}$ to  $M^{\mgl}_{k}$. 
For $S\subset \p$ we will also consider the corresponding $\shmglam(-)$.

Now assume that $S$ contains $p$ if $p>0$ (the author is not  sure whether this is really necessary).  

Then there are three possible ways of proving Theorem \ref{tbach}  with $M_{k,\lam}$ replaced by $M_{k,\lam}^{\mgl}$ (and for the corresponding homotopy $t$-structure for $\shmglam(k)$). 

Firstly, one may prove that in all the results of  \S\ref{sthom} one may replace $\dmlam(-)$ by $\shmglam(-)$. The key points here are the following ones: $\mgl\in \sht(k)_{\tsh\ge 0}$ (see Corollary 3.9 of \cite{hoycobord}); the corresponding analogue of  Proposition \ref{peff}(\ref{ieffdual})  is given by Theorem 5.2.6 of \cite{morintao}, whereas the  $\mgl$-analogue of 
 Proposition \ref{pteff}(4) follows from the vanishing of  
$\sht(F)(S^0, \mgl\{j\}[i])$ for any $i\in \z$, $j<0$, and any perfect field $F$ (that follows from  Theorem 8.5 of \cite{hoycobord}). 
Having these statements one can easily verify that the corresponding analogue of Theorem \ref{tineffkerm} is valid as well (probably one can also deduce this statement from Lemma 7.10 of ibid.). 
The latter allows to deduce the result in question from its $\dmlam(-)$-analogue (i.e., from Theorem \ref{tbach}) immediately.

Another possibility is to deduce the $\shmglam$-analogue of  Theorem  \ref{tbach} from the corresponding version of Theorem 12 of \cite{bachcons}; the latter statement  easily follows from Remark 4.3.3 of \cite{bondegl}.


Lastly, one hopes (see Remark 1.3.4 of \cite{bondegl}) that there exists an  (important) commutative diagram
$$\begin{CD}
 \shlam(k)@>{M^{\mgl}_{k,\lam}}>>\shmglam(k)\\
@VV{}V@VV{}V \\
D_{\afo,\lam}(k)@>{M^D_{k,\lam}}>>\dmlam(k)
\end{CD}$$
of functors; here we use the notation of   (Example 1.3.3 of)  
 ibid. for the lower left hand corner of this diagram; cf. also \S1 of \cite{bachcons}. 
  Certainly, this conjecture implies its "compact" analogue, 
 and it can be applied to the study of the conservativity of connecting functors. 

This diagram (along with Proposition 2.3.7 of \cite{bondegl}) 
 also implies that one may replace all $M_{-,\lam}$ in our statements by $M^D_{-,\lam}$. 


\end{document}